\documentclass[preprint,11pt]{elsarticle}
\usepackage{amsmath, amsthm, amsfonts, enumerate, mathrsfs, pmat}
\usepackage{amsmath,tikz, wasysym}

\newtheorem{theorem}{Theorem}[section]

\newtheorem{lemma}[theorem]{Lemma}
\newtheorem{proposition}[theorem]{Proposition}

\theoremstyle{definition}
\newtheorem{definition}[theorem]{Definition}
\theoremstyle{remark}
\theoremstyle{definition}
\newtheorem{example}[theorem]{Example}

\newcommand{\npmatrix}[1]{\left( \begin{matrix} #1 \end{matrix} \right)}

\newcommand{\C}{\mathbb{C}}
\newcommand{\F}{\mathbb{F}}
\newcommand{\mc}{\mathcal}

 \definecolor{helena}{rgb}{.2,.8,.4}
    \definecolor{olga}{rgb}{.8,.2,.2}
   \definecolor{todo}{rgb}{.2,.2,.8}

\journal{Linear Algebra and its Applications}

\begin{document}

\allowdisplaybreaks

\begin{frontmatter}

\title{The effect of assuming the identity as a generator on the length of the matrix algebra }

\author[UCD]{Thomas Laffey}
\ead{thomas.laffey@ucd.ie}
\author[MSU]{Olga Markova\fnref{fn2}}
\ead{ov\_markova@mail.ru}
\author[UCD]{Helena \v{S}migoc\corref{cor1}}
\ead{helena.smigoc@ucd.ie}

\cortext[cor1]{Corresponding author}
\fntext[fn1]{This work was supported by Science Foundation Ireland under Grant 11/RFP.1/MTH/3157.}
\fntext[fn2]{The author  is partially financially supported by grant MD-962.2014.1.}

\address[UCD]{ School of Mathematical Sciences, University College Dublin, Belfield, Dublin 4, Ireland
}
\address[MSU]{Department of Algebra, Faculty of Mechanics and Mathematics, Lomonosov Moscow State University,
Moscow 119991,   Russia}

\begin{abstract}
Let $M_n(\F)$ be the algebra of $n \times n$ matrices and let $\mc S$ be a generating set of $M_n(\F)$ as an $\F$-algebra. The length of a finite generating set $\mc S$ of  $M_n(\F)$ is the smallest number $k$ such that words of length not greater than $k$ generate $M_n(\F)$ as a vector space. Traditionally the identity matrix is assumed to be automatically included in all generating sets $\mc S$ and counted as a word of length $0$. In this paper we discuss how the problem changes if this assumption is removed.
\end{abstract}


\end{frontmatter}

\emph{This paper is dedicated to the memory of Hans Schneider.}

\section{Introduction}\label{sec:Introduction}

Let $\mc A$ be a finite-dimensional $\F$-algebra for some field $\F$. The {\em length of a finite generating set} $\mc S$ of algebra $\mc A$ is the smallest number $k$ such that words of length not greater than $k$ generate $\mc A$ as a vector space. The {\em length of the algebra} is the maximum of lengths of its generating sets.  The problem of evaluating the length of the full matrix algebra was posed by A. Paz in \cite{MR740668}, where it is conjectured that the length of the full matrix algebra $M_n(\F)$ is equal to $2n-2$. While several results on the length of $M_n(\F)$ \cite{MR1483779,MR2230653, MR2540352}, its generating sets \cite{MR2050111,MR2112883,MR2823023} and subalgebras \cite{ MR2366242, MR2494664, MR3032290,Magic} were found, the conjecture remains open.

It is usually assumed, that the identity matrix is included in generating set $\mc S$ and is regarded as having length $0$.  In this paper we pose the question, how do the lengths of  generating sets and of algebras differ if we remove this assumption. While we can talk about this notion for more general algebras, we will concentrate our discussion on matrix algebras.

Let us begin by introducing some notation.
Let $\mc S$ be a subset of an algebra $\mc A$, then $\mc L(\mc S)$ will denote the subalgebra of $\mc A$  generated by the set $\mc S$, and $\mc L_k(\mc S)$ will denote the vector space generated by all the words in $\mc S$ of length at most $k.$  If $\mc A$ is unitary then $\mc L_0(\mc S)=\F$ and $\mc L_0(\mc S)=0$ otherwise. Clearly we have:
 $$\mc L(\mc S)=\bigcup_{i=1}^{\infty}\mc L_i(\mc S).$$
The length of a generating set $\mc S$ will be denoted by $l(\mc S)$, and the length of the algebra $\mc A$ will be denoted $l(\mc A)$. Note that $l(\mc S)$ is the smallest number $k$, such that $\mc L_k(\mc S)=\mc L_{k+1}(\mc S).$

Now we introduce the same notions as above again, but this time we do not assume that the identity is included in the generating set $\mc S$.
 For a subset $\mc S$ of an algebra $\mc A$ we will denote $\mc L_k^0(\mc S)$ to be the vector space generated by all the words in $\mc S$ of length at most $k$, where we do not assume the identity is included in $\mc S.$ (So in general $\mc L_k^0(\mc S) \subseteq \mc L_k(\mc S)$.)
Similarly, $\mc L^0(\mc S)$ will denote the algebra generated by the set $\mc S.$
 We also define the notion of the length in this sense.  The length of a generating set $\mc S$ without the identity automatically included in $\mc S$ will be denoted by $l_0(\mc S)$ and $l_0(\mc A)$ will denote the length of $\mc A$ in this sense.

 We will also need some notation for matrices. $J_k(\lambda)$ will denote the $k \times k$ Jordan block corresponding to the eigenvalue $\lambda$, and $J_k=J_k(0).$ We will denote by $I_n$ the $n \times n$ identity matrix, by $O_n$, $O_{r,s}$ the  $n\times n$ and $r \times s$ zero matrices and by $E_{i,j}$ the matrix unit, i.e. the matrix with $1$ in $(i,j)$-th position and $0$ elsewhere.

\section{Main Results}

The following proposition is perhaps the main motivation, why the identity is traditionally assumed to be included the generating set.

\begin{proposition}
Let $\mc S$ be a generating set of a simple algebra $\mc A$ of dimension greater than $1$,  i.e. $\mc L(\mc S)=\mc A$. Then $\mc L^0(\mc S)=\mc A$. \end{proposition}

\begin{proof}
Since $\mc L^0(\mc S)$ is a nonzero ideal in $\mc A$ the statement follows.
\end{proof}

Let $\mc S$ be a generating set of $M_n(\F)$. Then, clearly
$l_0(\mc S)=l(\mc S)$ or  $l_0(\mc S)= l(\mc S)+1.$
Simple examples below show that both eventualities can occur.

\begin{example}
Let us consider a generating set
$$\mc S=\{E_{12},E_{21},E_{22}\} \subset M_2(\F).$$
Then $\mc L(\mc S)=M_2(\F)$, $l(\mc S)=1$ and $l_0(\mc S)=2.$
For $\mc S'=S \cup \{E_{11}\}$ we have $l(\mc S)=1$ and $l_0(\mc S)=1.$
\end{example}

 Next set of results preset some cases, for which  $l(\mc S)=l_0(\mc S).$

\begin{proposition}\label{prop:C-H}
Let $\mc S$ be a generating set for $M_n(\C)$.
\begin{enumerate}
\item If $l(\mc S)\geq n$ and if $\mc L_1^0(\mc S)$ contains an invertible matrix, then $l(\mc S)=l_0(\mc S).$
\item If $l(\mc S)=2n-2$ and if $\mc L_2^0(\mc S)$ contains a derogatory invertible matrix, then $l(\mc S)=l_0(\mc S).$
\end{enumerate}
\end{proposition}

\begin{proof}
Let $A\in M_n(\C)$ be an invertible matrix in $\mc L_1^0(\mc S)$. By the Cayley-Hamilton theorem:
$A^n+a_1A^{n-1}+\ldots+a_nI_n=0$, where for $i=1,\ldots, n:$ $a_i \in \C$ and $a_n \neq 0.$
Therefore:
 $$I_n=-\frac{1}{a_n}A^n-\frac{a_1}{a_n}A^{n-1}-\ldots-\frac{a_{n-1}}{a_n}A$$
 is contained in $\mc L_n^0(\mc S).$

To prove the second item, let $A$ be a derogatory invertible matrix in $\mc L_2^0(\mc S)$. Its minimal polynomial has degree at most $n-1$ so: $$A^{k}+a_1A^{k-1}+\ldots+a_k I_n=0$$ for $k \leq n-1$ and $a_k \neq 0.$ Now
$$I_n=-\frac{1}{a_k}A^{k}-\frac{a_1}{a_k}A^{k-1}-\ldots-\frac{a_{k-1}}{a_{k}}A$$
is contained in $\mc L_{2n-2}^0(\mc S).$
\end{proof}

The two lengths are the same for the matrix algebras $M_2(\F)$ and $M_3(\F).$ The equations $l(M_2(\F))=2$ and $l(M_3(\F))=4$ follow  from \cite{MR740668}.

\begin{theorem}
$l_0(M_2(\F))=2.$
\end{theorem}

\begin{proof}
Let $\mc S$ be a generating set for $M_2(\F)$ that does not contain $I_2$. Trivially we can replace it by an irredundant subset which generates $M_2(\F)$, so we can assume $\mc S$ is irredundant. Then $\mc S$ has either $2$ or $3$ elements. If $\mc S$ has three elements, then $\mc S \cup \{I_2\}$ is a spanning set for $M_2(\F)$ and since then $\mc L_1(\mc S)=M_2(\F)$, the result holds.

Now suppose that $\mc S$ has two elements, say $\mc S=\{A,B\}.$ If $\mc L_1^0(\mc S)$ contains an invertible matrix, then the result is true by Proposition \ref{prop:C-H}. So, arguing by contradiction, we may assume $A^2=\alpha A$ and $B^2=\beta B$, for some $\alpha, \beta \in \F$. If $\alpha=0$, then we may without loss of generality assume that
 $$A=\npmatrix{0 & 1 \\ 0 & 0}.$$
Now let
$$B=\npmatrix{b_{11} & b_{12} \\ b_{21} & b_{22}}.$$
Since $\mc L_1^0(\mc S)$ does not contain an invertible matrix, we have $\det(xA+B)=0$ for every $x \in \F$, and this implies, that $b_{21}=0$. This contradicts the fact that $\mc S$ generates $M_2(\F).$ So suppose that $\alpha \neq 0$. Replacing $A$ by $\frac{1}{\alpha}A$ we may assume that $A^2=A$ and using similarity we may assume that
 $$A=\npmatrix{1 & 0 \\ 0 & 0}.$$
Now $\det(x A+B)=0$ for all $x \in \F$ implies $b_{22}=0$ and since $\mc S$ generates $M_2(\F)$, we must have $b_{21}b_{12} \neq 0.$ But then $\det B \neq 0$ contradicting our earlier assumption.

\end{proof}

\begin{theorem}\label{cor: 3x3}
$l_0(M_3(\F))=4$.
\end{theorem}

\begin{proof}
It is sufficient to prove that $l_0(\mc S)\leq 4$ for any generating set $\mc S$.
If $\mc S$ contains an invertible matrix, then the result follows from Proposition \ref{prop:C-H}.

Next we assume that $\mc S$ contains a singular nonderogatory matrix $A$. We will distinguish several cases depending on the size of the Jordan block associated with $0$. First consider the case when  the Jordan canonical form of $A$ is $J_3$ or  $A$ has a singular Jordan block of size $2$ and an eigenvalue $a\in \F\setminus \{0\}$.  In this case we may without loss of  generality assume that $A=J_3$  or  $A=\npmatrix{0 & 1 & 0 \\ 0 & 0 & 0\\ 0 & 0 & a}.$   First consider the sequence of inclusions:
 $$\F A^2 \subseteq \mc L_1(\mc S)A^2 \subseteq \ldots \subseteq \mc L(\mc S)A^2=M_3(\F)A^2.$$
We have $A^2=E_{13}$ or $A^2=a^2E_{33}$,  i.e. $A^2$ is a rank one matrix, thus this sequence stabilises at most on the step $2$, so we get:
 $\mc L_2(\mc S)A^2=M_3(\F)A^2.$ Furthermore, $\mc L_2(\mc S)A^2 \subseteq \mc L_4^0(\mc S),$ hence we have proved that all matrices with the first two columns equal to zero are contained in $\mc L_4^0(\mc S).$ We repeat the inclusion argument for $A$, yielding that $\mc L_4^0(\mc S)$ contains all matrices with the first column equal to zero.
Now we choose $X=(x_{ij}) \in \mc S$ with $x_{31} \neq 0$  or (if such $X \in \mc S$ does not exists), we choose  $X=(x_{ij}) \in \mc S$ with $x_{21} \neq 0$ and $x_{31}=0.$ Since $\mc S$ is a generating set for $M_3(\F)$ such a matrix $X$ exists. Then $\mc L_{2}(\mc S)X$ projected to column $1$ is $\F^3$ and since $\mc L_{2}(\mc S)X\subseteq \mc L_{3}^0(\mc S)$, we conclude
 $M_3(\F) \subseteq \mc L_{4}^0(\mc S)$
as asserted.


If $A$ has a singular Jordan block of order $1$, then we may assume that it is of the form:
 $$A=\npmatrix{0 & 0 \\ 0 & B},$$
 where $B$ is a $2 \times 2$ nonsingular matrix. It follows from the proof of \cite[Proposition 3.19]{MR3032290} that it would not affect the length of a given generating set if we take an extension of $\F$ as the base field. Therefore we may assume that $B$ is either equal to $J_2(a)$ or to a diagonal matrix with two distinct elements $a$ and $b$.
  In the first case we notice that
  $$\npmatrix{0 & 1 \\ 0 & 0}=\frac{1}{a}B^2-B \in \mc L_2^0(\{B\})$$
   and that
   $$\npmatrix{1 & 0 \\ 0 & 1}=-\frac{1}{a^2}B^2+\frac{2}{a} B \in \mc L_2^0(\{B\}).$$
  In the second case we have
 $$\npmatrix{0 & 0 \\ 0 & 1}=\frac{a}{b(a-b)}B-B^2 \in \mc L_2^0(\{B\})$$
   and
   $$\npmatrix{1 & 0 \\ 0 & 0}=\frac{b}{a(b-a)}B-B^2 \in \mc L_2^0(\{B\}).$$
  In both cases we may use the inclusion argument to prove that $\mc L_4^0(\mc S)$ contains all matrices with the first column equal to zero. As we did before, choosing an $X \in \mc S$ with $x_{j1},$ $j \in \{2,3\}$, completes the proof.

From now on we may assume that all the elements in $\mc L_1^0(\mc S)$ are singular and derogatory. This implies that for all $A \in \mc S$ we have $A^2=\alpha(A)A$, for some $\alpha(A) \in \F.$  In  \cite{MR636208,MR709356} it is proved, that in this case $\mc S$ has to contain at least three elements in order to generate $M_3(\F).$

First let us assume that there exists $A \in \mc S$ with $\alpha(A)=0$. Without loss of generality we may assume that:
 $$A=\npmatrix{0 & 1 & 0 \\ 0 & 0 & 0 \\ 0 & 0 & 0}.$$
 Let $B$ be another element in $\mc S.$ Then:
  $$(A+B)^2=\alpha(A+B)(A+B),$$
  and this implies that
  $$AB+BA=\alpha(A+B)A+(\alpha(A+B)-\alpha(B))B.$$
Taking $B=(b_{ij})$ and using the form of $A,$ this gives us:
 $$\npmatrix{b_{21} & b_{22}+b_{11} & b_{23} \\ 0 & b_{21} & 0 \\ 0 & b_{31} & 0}=
\beta\npmatrix{0 & 1 & 0 \\ 0 & 0 & 0 \\ 0 & 0 & 0}+\gamma
\npmatrix{b_{11} & b_{12} & b_{13} \\ b_{21} & b_{22} & b_{23} \\ b_{31} & b_{32} & b_{33}},
 $$
 where $\beta=\alpha(A+B)$ and $\gamma=(\alpha(A+B)-\alpha(B))$. Now, if $\gamma
\neq 0,$ then $b_{21}=b_{31}=0,$ and the same is true if $\gamma=0$. This proves, that all matrices in $\mc S$ are of the form:
 $$\npmatrix{b_{11} & b_{12} & b_{13} \\ 0 & b_{22} & b_{23} \\ 0 & b_{32} & b_{33} },$$
 and this contradicts our assumption that $\mc S$ generates $M_3(\F).$

 Now we can assume that $\alpha(A) \neq 0$ for all $A \in \mc S.$ 
 First consider the case when there exists a matrix of rank $1$ in $\mc L_1^0(\mc S)$.
 Without loss of generality we may assume that there exists
  $$A=\npmatrix{1 & 0 & 0 \\ 0 & 0 & 0 \\ 0 & 0 & 0} \in \mc S.$$
  Let $B$ be another matrix in $\mc S.$ We again look at $(A+B)^2=\alpha(A+B)(A+B)$:
  $$\npmatrix{2b_{11} & b_{12} & b_{13} \\ b_{21} & 0 & 0 \\ b_{31} & 0 & 0}=
 \beta\npmatrix{1 & 0 & 0 \\ 0 & 0 & 0 \\ 0 & 0 & 0}+\gamma
\npmatrix{b_{11} & b_{12} & b_{13} \\ b_{21} & b_{22} & b_{23} \\ b_{31} & b_{32} & b_{33}},
 $$
 where $\beta=\alpha(A+B)-1$ and $\gamma=(\alpha(A+B)-\alpha(B)).$ If $\gamma=0$, then the above formula implies $b_{12}=0,$ $b_{13}=0$, $b_{21}=0$ and $b_{31}=0.$  If $\gamma
 \neq 0,$ then we have $b_{22}=0$, $b_{23}=0$, $b_{32}=0$ and $b_{33}=0.$ In this case $B^2=\alpha(B)B,$ gives us $b_{21}b_{12}=0,$ $b_{21}b_{13}=0,$ $b_{31}b_{12}=0$ and $b_{31}b_{13}=0.$ Now, if $b_{21} \neq 0$ or if $b_{31} \neq 0,$ then $b_{12}=0$ and $b_{13}=0,$ otherwise $b_{21}=0$ and $b_{31}=0.$ With this we have proved that all the matrices in $\mc L_1^0(\mc S)$ have one of the following forms:
  \begin{equation}\label{eq:block}\npmatrix{b_{11} & b_{12} & b_{13} \\ 0 & b_{22} & b_{23} \\ 0 & b_{32} & b_{33} }
    \text{ or }
     \npmatrix{b_{11} & 0 & 0 \\ b_{21} & b_{22} & b_{23} \\ b_{31} & b_{32} & b_{33} }.\end{equation}
But since for any two matrices $B$ and $C$ in $\mc S$ matrices $B$, $C$ and $B+C$ have to have one of the forms above, they all have to have the same form. This contradicts our assumption that $\mc S$ generates $M_3(\F).$

Now consider the case that all matrices in $\mc L_1^0(\mc S)$ have rank $2$.
Without loss of generality we may assume that there exists
  $A=\npmatrix{0 & 0 & 0 \\ 0 & 1 & 0 \\ 0 & 0 & 1} \in \mc S.$   Then $I_3-A=E_{11}$.
Since all matrices from $\mc L_1^0(\mc S)$ are derogatory, then all the matrices from the space $\mc L_1(\mc S)$ share this property, therefore  all the matrices from the space $\mc L_1^0((\mc S\setminus\{A\})\cup\{I_3-A\})$ also are derogatory.
  
  Let $B$ be another matrix in $\mc S.$ As above we consider $((I_3-A)+B)^2=\alpha(I_3-A+B)(I_3-A+B)+\delta I_3$:
  $$\npmatrix{2b_{11} & b_{12} & b_{13} \\ b_{21} & 0 & 0 \\ b_{31} & 0 & 0}=
 \beta\npmatrix{1 & 0 & 0 \\ 0 & 0 & 0 \\ 0 & 0 & 0}+\gamma
\npmatrix{b_{11} & b_{12} & b_{13} \\ b_{21} & b_{22} & b_{23} \\ b_{31} & b_{32} & b_{33}}+\delta\npmatrix{1 & 0 & 0 \\ 0 & 1 & 0 \\ 0 & 0 & 1},
 $$
 where $\beta=\alpha(I_3-A+B)-1$, $\gamma=(\alpha(I_3-A+B)-\alpha(B)).$ If $\gamma=0$, then the above formula implies $b_{12}=0,$ $b_{13}=0$, $b_{21}=0$ and $b_{31}=0.$  
 If $\gamma
 \neq 0,$ then we have   $b_{23}=0$ and $b_{32}=0$, $b_{22}=b_{33}$.  If $b_{13}\neq 0$, in this case $B^2=\alpha(B)B,$ gives us $(B^2)_{23}=b_{21}b_{13}=0,$ i.e. $b_{21}=0$, $(B^2)_{22}=(B^2)_{33}$, i.e. $b_{31}=0$. 
 Analogously, if $b_{31}\neq 0$, then  $B^2=\alpha(B)B$ implies that $b_{12}=b_{13}=0$. 
If both $b_{13}=b_{31}=0$, then $B^2=\alpha(B)B$ gives us that  $(B^2)_{22}=b_{12}b_{21}+b_{22}^2=b_{22}^2=(B^2)_{33}$, thus   $b_{12} b_{21}=0$.
Again, we have proved that all matrices in $\mc S$ have block forms as in~\eqref{eq:block}, a contradiction.

\end{proof}

Next we show that even generating sets $\mc S$ of $M_n(\F)$ with relatively large lengths can have $l_0(\mc S)=l(\mc S)+1$.

\begin{definition}
For $k=-(n-1), -(n-2),\ldots,n-2, n-1$ we will say the the elements of the form $a_{i i+k}$ form the \emph{$k-$diagonal} of a matrix $A \in M_n(\F).$
\end{definition}

\begin{lemma}\label{lem:diagonal}
Let $A$ be a matrix with all its nonzero elements on its $k$-diagonal, and let  $B$ be a matrix with all its nonzero elements on its $l$-diagonal. Then $AB$ can have nonzero entries (if any) only on its $k+l$-diagonal.
\end{lemma}

\begin{theorem}\label{thm: Nonderogatory}
Let $B_n=E_{n-1,1}-E_{n,2} \in M_n(\C),$ $n \geq 3.$ Then $$\mc L(\{J_n,B_n\})=M_n(\C),$$ $l(\{J_n,B_n\})=2n-3$ and $l_0(\{J_n,B_n\})=2n-2.$
\end{theorem}

\begin{proof}
Since $B_n^2=0$, every word in $J_n$ and $B_n$ is of the form:
 \begin{equation}\label{word}
 B_n^{i_1}J_n^{j_1}B_nJ_n^{j_2}B^{i_2}\ldots B_nJ_n^{j_s}B_n^{i_2},
 \end{equation}
 where $i_1,i_2 \in \{0,1\},$ and $j_l$, $l=1,2,\ldots,s$, are positive integers.
Lemma \ref{lem:diagonal} tells us that a word of the form (\ref{word}) can have nonzero entries (if any) only on $\sum_{l=1}^sj_s-(n-2)(i_1+i_2+s-1 )$-diagonal.

First we show that $l_0(\{J_n,B_n\}) > 2n-3.$ To show this, we consider the $0$-diagonal of matrices in $\mc L_{2n-3}^{0}(\{J_n,B_n\}).$ For a word (\ref{word}) to have entries on the $0$-diagonal, we need to have:
  $$\sum_{l=1}^sj_s-(n-2)(i_1+i_2+s-1 )=0,$$
  hence at least one $B_n$ has to appear in the product. All the products with nonzero entries on the $0$-diagonal involving only one $B_n$ are of the form:
   $$J_n^jB_nJ_n^{n-2-j} \text{ for }j=0,1,\ldots,n-2.$$
There are $n-1$ such products, and since the $0$-diagonal has dimension $n$, we need at least one word involving $2$ matrices $B_n$, in order to have the complete $0$-diagonal included in the span. For this word to have elements on $0$-diagonal we need
 $\sum_{l=1}^sj_s-(n-2)(2)=0,$
 hence $\sum_{l=1}^sj_s=2n-4.$ The length of this word is $2n-2$, and this proves $l_0(\{J_n,B_n\}) > 2n-3.$

Now we show that $l(\{J_n,B_n\})=2n-3,$ which also implies that  $l_0(\{J_n,B_n\})=2n-2.$

Let $\mc V_n(k)$, $k=-(n-1),-(n-2),\ldots,(n-2),(n-1)$, denote the vector space of $n\times n$ matrices with nonzero entries on $k$-diagonal. Note that  $\dim(\mc V_n(-(n-k)))=\dim(\mc V_n(n-k))=k$
 for $k=1,\ldots,n.$ For each of the vector sets $\mc V_n(k)$ we will find a basis $\mc B_n(k)$ consisting of words in $J_n$ and $B_n$ of length less than or equal to $2n-3.$

 First we compute $B_nJ_n^{n-3}B_n=-E_{n,1}$, and this tells us that
 $$\mc B_n(-(n-1))=\{B_nJ_n^{n-3}B_n\}$$
 is a basis for $\mc V_n(-(n-1)).$ Next we compute:
  $$B_nJ_n^{n-3}B_nJ_n^{k-1}=-E_{n,k}$$
for $k=2,\ldots,n-1$ and for $s \leq n-2$, $t \leq n-2$:
 $$J_n^sB_n=E_{n-s-1,1}-E_{n-s,2} \text{ and }J_n^{s}B_nJ_n^{t}=E_{n-s-1,t+1}-E_{n-s,t+2}.$$
 Now it is clear that
\begin{align*}
\mc B_n(-(n-k))&=\{B_nJ_n^{n-3}B_nJ_n^{k-1}\}\cup \{J_n^lB_nJ_n^{k-2-l};\ l=0,\ldots,k-2\}\\
&=\{-E_{n,k}\}\cup \{E_{n-l-1,k-l-1}-E_{n-l,k-l};\ l=0,\ldots,k-2\}
\end{align*}
is a basis for $\mc V_n(-(n-k))$ for $k=2,\ldots,n-1.$

To form $\mc B_n(0)=\mc B_n(n-n)$ we replace the word  $B_nJ_n^{n-3}B_nJ_n^{n-1}$ of length $2n-2$ with $I_n$:
\begin{align*}
\mc B_n(0)&=\{I_n\}\cup \{J_n^lB_nJ_n^{n-2-l};\ l=0,\ldots,n-2\}\\
&=\{I_n\}\cup \{E_{n-l-1,n-l-1}-E_{n-l,n-l};\ l=0,\ldots,n-2\}
\end{align*} Note, that the set $\mc B_n(0)$ is linearly independent over any field of characteristics $0$.

To get a basis for $\mc V_n(n-k)$ we also need to compute:
 $$J_n^{n-1}BJ_n^s=-E_{1,s+2}$$
for $s=0,\ldots, n-3.$ Now it is clear that
\begin{align*}
 \mc B_n(n-k)&=\{J_n^lB_nJ_n^{2n-k-2-l};l=n-k,\ldots,n-1\}\\
 &=\{E_{n-l-1,2n-k-l-1}-E_{n-l,2n-k-l};l=n-k,\ldots,n-2 \}\cup\{-E_{1,n-k+1}\}=\\
 &=\{E_{k-1-j,n-j-1}-E_{k-j,n-j};j=0,\ldots,k-2\}\cup\{-E_{1,n-k+1}\}
 \end{align*}
 is a basis for $\mc V_n(n-k)$ for $k=2,\ldots,n-1$. 
 Finally we take $$\mc B_n(n-1)=\{J_n^{n-1}\}.$$
 
 Since all the words in $\mc B_n(l),$ $l=-(n-1),-(n-2), \ldots, (n-2),(n-1)$ have length less than or equal to $2n-3$, we have proved our claim.
 \end{proof}

 We conclude this paper by exhibiting a class of generating sets for $M_n(\C)$ with two elements, akin to the one in Theorem \ref{thm: Nonderogatory}.

\begin{theorem}\label{thm:Ji} Let  $i< n$ be positive integers.
$\mc L(\{J_n^i,(J_n^T)^{n-i}\})=M_n(\C)$ if  and only if $n-i$ and $i$ are co-prime.  If $\gcd(n,i)=d$,  then $\mc L(\{J_n^i,(J_n^T)^{n-i}\})\cong M_{\frac n d}(\C)$.
\end{theorem}

\begin{proof}
We will show that, for $i$ and $n-i$ co-prime,  $\mc L(\{J_n^i,(J_n^T)^{n-i}\})$ has no invariant subspaces except $\{0\}$ and $\C^n$, thereby showing that $\mc L(\{J_n^i,(J_n^T)^{n-i}\})=M_n(\C)$ by Burnside's Theorem.

Let
\begin{equation}\label{vector}
v=\npmatrix{a_1 & a_2 & \ldots & a_k & 0 & \ldots & 0}^T \in \C^n
\end{equation}
with $a_k \neq 0.$ If $k \leq i$, then $(J_n^T)^{n-i}v$ has a nonzero element in $(k+n-i)$-th position, and if $k \geq i+1$, then $J_n^{i}v$ has a nonzero element in $(k-i)$-th position.
We consider the following recursively defined sequence with $k_0=k$ and:
\begin{equation}\label{k-sequence}
k_{t+1}=\begin{cases}
            k_t+(n-i) &\text{ for } k_t \leq i \\
            k_t-i  &\text{ for } k_t \geq 1+i
            \end{cases}.
            \end{equation}
Let $j_0$ be the smallest integer at which the sequence repeats: $k_{i_0}=k_{j_0}$ for some $i_0 <j_0.$ Then:
$k_{j_0}=k_{i_0}-\alpha i+\beta (n-i)$ and $\alpha i=\beta (n-i).$ Since $i$ and $n-i$ are co-prime, we need $\alpha=\gamma (n-i)$ and $\beta=\gamma i$. This tells us that the sequence repeats after exactly $n$ steps, independent of where we start, and it contains all numbers $\{1,2,\ldots,n\}$ is some order.

Let $\mc U$ be an invariant subspace for $\mc L(\{J_n^i,(J_n^T)^{n-i}\})$, and let $v$ be a vector of the form (\ref{vector}) in $\mc U$ with $a_k \neq 0$. Since $\mc U$ has to contain an eigenvector of $J_n^i$ we know that such a vector exists in $\mc U$ with $k\leq i$. Let $k_i$ be the sequence defined in (\ref{k-sequence}). Let $v_0=v$ and let
\begin{equation}\label{v-sequence}
v_{t+1}=\begin{cases}
            (J_n^T)^{(n-i)}v_t &\text{ if } k_t \leq i \\
            (J_n)^iv_t  &\text{ if } k_t \geq 1+i
            \end{cases}.
            \end{equation}
Notice that $v_i$ has a nonzero element $a_k$ in the $k_i$-th position followed by zero elements. Taking into account our argument about the sequence $k_i$ above, we conclude that $\{v_i; i=0,\ldots,n-1\}$ forms a basis for $\C^n$, thus proving $\mc U=\C^n.$


Now assume that $\gcd(i,n-i)=d$, $d\neq 1$. We can consider matrices    $J_n^i,(J_n^T)^{n-i}$ as block $\frac n d\times \frac n d$ matrices with blocks from $M_d(\C)$. We have all blocks in $J_n^i$ and $(J_n^T)^{n-i}$ equal to either $O_d$ or $I_d$, hence $$J_n^i=J^{\frac i d}_{\frac n d}\otimes I_d,\ (J_n^T)^{n-i}=(J_{\frac n d}^T)^{\frac {n-i} {d}}\otimes I_d.$$ Therefore the algebra  $\mc L(\{J_n^i,(J_n^T)^{n-i}\})$ is naturally  isomorphic to  $\mc L(\{J^{\frac i d}_{\frac n d}, (J_{\frac n d}^T)^{\frac {n-i}{d}}\})$ which is $M_{\frac n d}(\C)$ as shown above.
\end{proof}

\begin{lemma}\label{lem:n-2}
Let $n,i$, be positive integers, $n\geq 2$, $i< n$, $\gcd(n,i)=1$. Let $J_n^i, (J_n^T)^{n-i}\in M_n(\C)$.
Denote $$\mc W(n,i)=\begin{cases} \mc V(n-2i)\oplus \mc V(-2i),\ {\mbox if}\ 2i<n,\\ \mc V(n-2i)\oplus \mc V(2n-2i),\ {\mbox if}\  2i>n,\\
\mc V(0),\ {\mbox if}\  i=1\ {\mbox and}\  n=2,\\\end{cases}$$ for the spaces $\mc V(j)$  introduced in Theorem~\ref{thm: Nonderogatory}.

Then the span of words in $J_n^i$ and $(J_n^T)^{n-i}$ of length $n-2$ is a proper subspace of the space $\mc W(n,i)$.
\end{lemma}

\begin{proof}
For the sake of simplicity let us denote $A_n=J_n^i, B_n=(J_n^T)^{n-i}$.

Take an arbitrary word in $A_n$ and $B_n$ of length $n-2$:
 \begin{equation}\label{word_ab}
 A_n^{p_1}B_n^{t_1}A_n^{p_2}B_n^{t_2}\ldots A_n^{p_s}B_n^{t_s},
 \end{equation}
where $p_1,t_s$ are nonnegative and $t_l$, $l=1,2,\ldots,s-1$, $p_l$, $l=2,\ldots,s$ are positive integers. Set $p=\sum_{l=1}^s p_l, t=\sum_{l=1}^s t_l$. Then $p+t=n-2$ and
Lemma \ref{lem:diagonal} implies that a word of the form (\ref{word_ab}) can have nonzero entries  only on  $((i-t)n-2i)$-diagonal (since $ip-(n-i)t= i(n-2-t)-(n-i)t= (i-t)n-2i)$).

If $2i=n$, then we need to have $i=1$ and $n=2$, and we obtain the word of length zero  $I_2$.

For $2i<n$ we have $|(i-t)n-2i|\leq n-1$ if and only if either $t=i$ or $t=i-1$. That is nontrivial  words of length $n-2$  have nonzero elements either on $(-2i)$- or $(n-2i)$-diagonal, thus are contained in  $\mc W(n,i)$.

For $2i>n$ we have $|(i-t)n-2i|\leq n-1$ if and only if either $t=i-2$ or $t=i-1$. That is nontrivial words of length $n-2$  have  elements either on $(n-2i)$- or $(2n-2i)$-diagonal, thus are also contained in  $\mc W(n,i)$.

Consequently, the span of all words in $A_n$ and $B_n$ of length $n-2$ is contained in
$\mc W(n,i)$.

Now we are left to prove that all words in $A_n$ and $B_n$ of length $n-2$ span a subspace in
$\mc W(n,i)$, which does not coincide with the whole space $\mc W(n,i)$.

We prove this statement using the induction on $n$.

Induction base:
If $n=2$, the only possible value for $i$ is $1$. Since $n-2=0$, the only word of length $n-2$ is thus $I_2$, and since $\dim \mc V(0)=n=2$, the base of induction  follows.

Induction step:
Let $n\geq 3$ and assume that for all $n'<n$ the statement holds.

I. If $i=1$ and $n\geq 3$, the result follows from the proof of~\cite[Theorem 2, case $p=n-1$]{MR2823023} for arbitrary $n$.

Consider $i\geq 2$.
Let us write $n$ as $n=qi+r$, where $1\leq r<i$.

II. Assume that $2\leq i< \frac n 2$.

In this case we have $B_n^2=0$ and $A_n^{q+1}=0$, and a nonzero  word  (\ref{word_ab}) is of  the form    $$
 A_n^{p_1}B_nA_n^{p_2}B_n\ldots A_n^{p_s},
$$ where $p_l\leq {q}$ for all $l=1,\ldots,s$ and $s=i$ or $s=i+1.$

Notice that all elements of a product $B_nA_n^j$, $j=0,\ldots,q$, are equal to zero, except those in the last $i$ rows.  The latter  have the following form:
\begin{align*}
&\begin{pmatrix}I_i&O_i&\ldots&O_i&O_{i,r}\end{pmatrix},\\ &\begin{pmatrix}O_i&I_i&\ldots&O_i&O_{i,r}\end{pmatrix},\\ &\vdots\\ &\begin{pmatrix}O_i&O_i&\ldots&I_i&O_{i,r}\end{pmatrix},\\ &\begin{pmatrix}O_i&O_i&\ldots&O_i&I_{i,r}\end{pmatrix},
 \end{align*}
  where $O_{i,r}$ and $I_{i,r}$ denote $i\times r$ submatrices of $O_i$ and $I_i$, respectively, formed by first $r$ columns.  Now we see that for $j<q-1$, the last $i$ columns of $B_nA_n^j$ are zero, and this gives us $$B_nA_n^jB_n=0,\   j=1,\ldots,q-2.$$
Therefore a word $V$ of length $n-2$, which is not trivially zero  and has its nonzero elements located on $(n-2i)$-diagonal, has the form  
$$A_n^{p_1}B_nA_n^{q-\varepsilon_1}B_n\cdots A_n^{q-\varepsilon_{i-2}}B_nA_n^{p_{i}},$$ $\varepsilon_j\in \{0,1\}$, $j=1,\ldots,i-2$.

Consider the projections of $B_nA_n^{q-1}$ and $B_nA_n^{q}$ on the lower right $i\times i$ corner. They are $(J^T_i)^r$ and $J_i^{i-r}$, respectively.
Therefore a subword $$B_nA_n^{q-\varepsilon_1}B_n\cdots B_nA_n^{q-\varepsilon_{i-2}}$$ of $V$ corresponds to a word of length $i-2$ in $(J^T_i)^r$ and $J_i^{i-r}$. By construction $\gcd(i,r)=1$, hence by induction hypothesis the span of all words of length $i-2$ in $(J^T_i)^r$ and $J_i^{i-r}$ is a proper subspace $\mc W$ of the space $\mc W(i,i-r)\subset M_i(\C)$. Take the words $U_1,\ldots, U_s,U_{s+1},\ldots,U_m$, which form a basis of the subspace $\mc W$, where $m<\dim \mc W(i,i-r)$, $U_1,\ldots,U_{s}$ are lower-triangular, $U_{s+1},\ldots,U_m$ are upper-triangular. They have pre-images  $$V_1=B_nA_n^{q-\varepsilon_{1,1}}B_n\cdots A_n^{q-\varepsilon_{i-2,1}},\ldots,  V_m=B_nA_n^{q-\varepsilon_{1,m}}B_n\cdots A_n^{q-\varepsilon_{i-2,m}},$$ $\varepsilon_{a,b}\in \{0,1\}$ for all $a=1,\ldots,i-2,$ $b=1,\ldots,m$.
The words $V_1B_n,\ldots, V_mB_n$ span a space equal to $\mc W$ located in the left lower $i\times i$ corner of $n\times n$ matrices. Finally we consider all words $A_n^{p_1}V_jB_nA_n^{p_{i}}$ of length $n-2$.

1. Suppose $2r<i$.  The intersection of $\mc V(n-2i)$ with the space generated by $i\times 2i$ matrices formed by rows $r+1,\ldots,r+i$ and columns $(q-2)i+1,\ldots, qi$  is isomorphic to $\mc W(i,i-r)$. To get nonzero entries in this submatrix we are able to take the words $A_n^{q-1}V_{h}B_nA_n^{q-2}$, $h\geq s+1$ and $A_n^{q-1}V_{g}B_nA_n^{q-1}$, $1\leq g\leq s$.  By construction \begin{multline*}\dim \langle A_n^{q-1}V_{h}B_nA_n^{q-2},\ A_n^{q-1}V_{g}B_nA_n^{q-1},  s+1\leq h\leq m, 1\leq g\leq s\rangle=m<\\<\dim \mc W(i,i-r),\end{multline*} thus the words in $A_n$ and $B_n$ of length $n-2$ do not span $\mc W(n,i)$.

2. Suppose $2r=i$, that is $r=1$, $i=2$.  The only nontrivial word of length $n-2$ with $1$ letter $B_n$ which lies in the $2$-dimensional intersection of $\mc V(n-2i)$ with the space generated by $2\times 2$ matrices formed by rows $r+1, r+2$ and columns $2q-1,  2q$  is $A_n^{q-1}B_nA_n^{q-1}$. Thus   the words in $A_n$ and $B_n$ of length $n-2$ do not span $\mc W(n,2)$.

3. Suppose $2r>i$.  The intersection of $\mc V(n-2i)$ with the space generated by $i\times (i+r)$ matrices formed by rows $r+1,\ldots,r+i$ and  last $i+r$ columns is isomorphic to $\mc W(i,i-r)$, since in $M_i(\C)$ we have $\dim \mc V(-2(i-r))=2r-i<r$. To get to this submatrix we are able to take the words $A_n^{q-1}V_{h}B_nA_n^{q-1}$, $h\geq s+1$ and $A_n^{q-1}V_{g}B_nA_n^{q}$, $1\leq g\leq s$.  By construction  \begin{multline*}\dim \langle A_n^{q-1}V_{h}B_nA_n^{q-1},\ A_n^{q-1}V_{g}B_nA_n^{q},  s+1\leq h\leq m, 1\leq g\leq s\rangle=m<\\<\dim \mc W(i,i-r), \end{multline*} thus the words in $A_n$ and $B_n$ of length $n-2$ do not span $\mc W(n,i)$.

III. If $i>\frac n 2$, then we apply the arguments from I or II to matrices $A_n^T$ and $B_n^T$, to show that their products of length $n-2$ span a proper subspace $\mc W$ in $\mc W(n,n-i)$. Thus we take transposes again, and see that $A_n$ and $B_n$ generate a proper subspace $(\mc W)^T$  in  $(\mc W(n,n-i))^T=\mc W(n,i)$ and the result follows.

\end{proof}

\begin{theorem} Let  $i< n$ be positive integers, $\gcd(n,i)=d$. Then for matrices from $M_n(\C)$ we have $$l(\{J_n^i,(J_n^T)^{n-i}\})=l_0(\{J_n^i,(J_n^T)^{n-i}\})=2\frac n d-2.$$
\end{theorem}

\begin{proof}
It follows from the proof of Theorem~\ref{thm:Ji}, that if $\gcd(n,i)=d>1$, then   $l(\{J_n^i,(J_n^T)^{n-i}\})=l(\{J^{\frac i d}_{\frac n d}, (J_{\frac n d}^T)^{\frac {n-i}{d}}\})$, therefore it is sufficient to consider the case when $i$ and $n$ are co-prime.

Trivially we have $l(\{J_n^i,(J_n^T)^{n-i}\})=l(\{J_n^i,J_n^i+(J_n^T)^{n-i}\})$.
Since $J_n^i+(J_n^T)^{n-i}$ is a non-derogatory matrix, which has $n$ distinct roots of unity as its eigenvalues, then the upper bound $l(\{J_n^i,J_n^i+(J_n^T)^{n-i}\})\leq 2n-2$  follows from \cite[Theorem 4.1. (b)]{MR1483779}.

Now we prove the  lower bound $l(\{J_n^i, (J_n^T)^{n-i}\})\geq 2n-2$.

Without loss of generality we assume now that $i<\frac {n}{2}$, since $$l(\{J_n^i,(J_n^T)^{n-i}\})=l( \{(J^T_n)^i,J_n^{n-i}\}).$$

Arguing similarly as when calculating the period of the sequence \eqref{k-sequence} in Theorem~\ref{thm:Ji}, and applying Lemma~\ref{lem:diagonal}  we calculate that a  word with $p$ letters $J_n^i$ and $t$ letters $(J_n^T)^{n-i}$  has its nonzero entries on $-2i$- or $(n-2i)$-diagonal, i.e. belongs to the space $\mc W(n,i)$, if and only if $p+t+2=\alpha n$ for positive integers $\alpha$. Lemma~\ref{lem:n-2} yields that considering only $\alpha=1$ is not enough to generate $\mc W(n,i)$, thus it is necessary to consider $\alpha=2$ as well. Consequently,  there exists a  word of length $2n-2$ which is not contained in $\mc L_{2n-3}(\{J_n^i,(J_n^T)^{n-i}\})$ and $l(\{J_n^i, (J_n^T)^{n-i}\})\geq 2n-2$.

Since $J_n^i+(J_n^T)^{n-i}$  is an  invertible matrix, the equation $l(\{J_n^i, (J_n^T)^{n-i}\})=l_0(\{J_n^i, (J_n^T)^{n-i}\})$ follows from  Proposition~\ref{prop:C-H}.
\end{proof}

One of Hans Schneider's first papers \emph{A pair of matrices with property P} \cite{MR0068509} deals with conditions which guarantee the simultaneous triangularizability of a set of matrices, and he continued to contribute to, and influence research in matrix theory from the standpoint of the theory of finite dimensional algebras.

\appendix
\section*{References}

\bibliographystyle{plain}
\bibliography{AlgebraLengths}

\begin{thebibliography}{10}

\bibitem{MR2112883}
D.~Constantine and M.~Darnall.
\newblock Lengths of finite dimensional representations of {PBW} algebras.
\newblock {\em Linear Algebra Appl.}, 395:175--181, 2005.

\bibitem{MR709356}
Fergus~J. Gaines, Thomas~J. Laffey, and Helene~M. Shapiro.
\newblock Pairs of matrices with quadratic minimal polynomials.
\newblock {\em Linear Algebra Appl.}, 52/53:289--292, 1983.

\bibitem{MR2494664}
A.~E. Guterman and O.~V. Markova.
\newblock Commutative matrix subalgebras and length function.
\newblock {\em Linear Algebra Appl.}, 430(7):1790--1805, 2009.

\bibitem{Magic}
A.~E. Guterman, O.~V. Markova, and S.D. Sochnev.
\newblock The algebra of semimagic matrices and its length.
\newblock {\em Zap. Nauch. Sem. POMI}, 419:52--76, 2013.
\newblock Translated in J. Math. Sci., New York, 2014, 199:4, 400--413.

\bibitem{MR636208}
Thomas~J. Laffey.
\newblock Algebras generated by two idempotents.
\newblock {\em Linear Algebra Appl.}, 37:45--53, 1981.

\bibitem{MR2540352}
M.~S. Lambrou and W.~E. Longstaff.
\newblock On the lengths of pairs of complex matrices of size six.
\newblock {\em Bull. Aust. Math. Soc.}, 80(2):177--201, 2009.

\bibitem{MR2050111}
W.~E. Longstaff.
\newblock Burnside's theorem: irreducible pairs of transformations.
\newblock {\em Linear Algebra Appl.}, 382:247--269, 2004.

\bibitem{MR2230653}
W.~E. Longstaff, A.~C. Niemeyer, and Oreste Panaia.
\newblock On the lengths of pairs of complex matrices of size at most five.
\newblock {\em Bull. Austral. Math. Soc.}, 73(3):461--472, 2006.

\bibitem{MR2823023}
W.~E. Longstaff and Peter Rosenthal.
\newblock On the lengths of irreducible pairs of complex matrices.
\newblock {\em Proc. Amer. Math. Soc.}, 139(11):3769--3777, 2011.

\bibitem{MR2366242}
O.~V. Markova.
\newblock Computation of the lengths of matrix subalgebras of a special form.
\newblock {\em Fundam. Prikl. Mat.}, 13(4):165--197, 2007.
\newblock Translated in J. Math. Sci., New York, 2008, 155:6, 908--931.

\bibitem{MR3032290}
O.~V. Markova.
\newblock The length function and matrix algebras.
\newblock {\em Fundam. Prikl. Mat.}, 17(6):65--173, 2011/12.
\newblock Translated in J. Math. Sci., New York, 2013, 193:5, 687--768.

\bibitem{MR1483779}
Christopher~J. Pappacena.
\newblock An upper bound for the length of a finite-dimensional algebra.
\newblock {\em J. Algebra}, 197(2):535--545, 1997.

\bibitem{MR740668}
Azaria Paz.
\newblock An application of the {C}ayley-{H}amilton theorem to matrix
  polynomials in several variables.
\newblock {\em Linear and Multilinear Algebra}, 15(2):161--170, 1984.

\bibitem{MR0068509}
Hans Schneider.
\newblock A pair of matrices with property {$P$}.
\newblock {\em Amer. Math. Monthly}, 62:247--249, 1955.

\end{thebibliography}
\end{document}